\documentclass[12pt,a4paper]{article}
\usepackage{amssymb,amsmath,amsthm}
\usepackage{graphicx}
\usepackage{color}

\theoremstyle{plain}
\newtheorem{theorem}{Theorem}[section]

\theoremstyle{definition}

\newtheorem{defn}[theorem]{Definition}

\newcommand\cref[1]{Corollary~\ref{cor:#1}}

\title{Finding non-minority balls with majority and plurality queries}

\author{Huilan Chang$^{a,b}$ \and D\'aniel Gerbner$^c$ \and Bal\'azs Patk\'os$^c$
\\
\small $^a$ Department of Applied Mathematics, National University of Kaohsiung, Taiwan, R.O.C. \\ 
\small $^b$ Scientific Multi-simulation Research Center, College of Science, \\
\small National University of Kaohsiung, Taiwan, R.O.C. \\
\small $^c$ Alfr\'ed R\'enyi Institute of Mathematics, Hungarian Academy of Sciences}

\begin{document}
\maketitle

\begin{abstract}
Given a set of $n$ colored balls, a \textit{majority, non-minority or plurality ball} is one whose color class has size more than $n/2$, at least $n/2$ or larger than any other color class, respectively. We describe linear time algorithms for finding non-minority balls using query sets of size $q$ of the following form: the answer to a majority/plurality query $Q$ is a majority/plurality ball in $Q$ or the statement that there is no such ball in $Q$.
\end{abstract}

\section{Introduction}

The basic instance of the Majority Problem is the following. We are given $n$ balls with an unknown 2-coloring. We can pick two balls and ask if they have the same color or not as a query. Our goal is to find a so-called majority ball, which is a ball of a color class of size greater than $n/2$ (i.e. the larger color class), or show that the color classes have the same size, using as few queries as possible. Saks and Werman \cite{SakWer91} showed that the minimum number of queries needed is $n-b(n)$, where $b(n)$ is the number of 1’s in the binary form of $n$. Earlier Fisher and Salzberg \cite{FisSal82} showed that if the number of colors is also unknown, then $\lceil 3n/2\rceil-2$ queries are needed and enough. In case of more colors, a natural question is to look for a ball of the largest color class instead of a majority ball (if exists). Such a ball is called a plurality ball. See \cite{Aig04,AigDeMMon05,GerKatPal13} for results on the plurality problem.

Another 
natural generalization is to consider larger queries. In this case there are several different ways how a query can be answered. They each define different variants of the majority problem. See \cite{Bor14,DeMKraWie12,EppHir18,gkppvw,glv2018,gv2018} for results in case of large queries.

In this paper we are going to study and generalize a variant introduced by Gerbner, Keszegh, P\'alv\"olgyi, Patk\'os, Vizer and Wiener \cite{gkppvw}. Before describing it in details, we introduce the following basic notions.


\begin{defn} In a set $X$ of balls, a ball $x\in X$ is a \textit{majority ball} if there are more than $|X|/2$ balls of the color of $x$ in $X$; $x$ is a \textit{non-minority ball} if there are at least $|X|/2$ balls of the color of $x$ in $X$. A ball $x$ is a \textit{plurality ball} if there are more balls of the color of $x$ in $X$, than of any other color. A ball $x$ is an \textit{almost plurality ball} if it is either a plurality ball, or a non-minority ball.
\end{defn}

Note that a ball is almost plurality but not plurality only if there are only two colors, $|X|$ is even and both color classes have size $|X|/2$.
Note furthermore that if there are only two colors in $X$, the majority and the plurality balls are the same.

Gerbner, Keszegh, P\'alv\"olgyi, Patk\'os, Vizer and Wiener \cite{gkppvw} considered the following variant. The answer to a query $Q$ is YES and a majority ball inside $Q$ if such a ball exists, and NO if there is no majority color inside $Q$. They observed that in the case of two colors, using these type of queries of odd size, one can never claim there is no majority ball, as the answers are always YES, thus we can never rule out the possibility that there is only one color. This means we cannot solve the majority problem in the original sense. However, they proved that one can always show a non-minority ball after $O(n)$ queries of size 3.

\begin{defn}
For a \textit{majority query} $Q$ the answer is NO if there is no majority ball in $Q$, otherwise the answer is YES and a majority ball $x$ is also presented.

For a \textit{plurality query} $Q$ the answer is NO if there is no plurality ball in $Q$, otherwise the answer is YES and a plurality ball $x$ is also presented.
\end{defn}

Sometimes we will call an answer consisting of YES and a ball simply YES answer, when we do not care which ball was shown.


We will consider the following functions. The first letter $M$ or $P$ denotes that we use majority or plurality queries. The second letter 
is $B$ if we want to find a non-minority ball in $X$, and $C$ if we try to find an almost-plurality ball. Then in brackets we write the size $|X|$ of the underlying set, the size of the queries and the number of colors.

Using the above notation the results of \cite{gkppvw} can be summarized as $MB(n,3,2)=O(n)$, $MB(n,2k,2)=O(n)$ and $MB(n,2k+1,2)=O(n^2)$. Note that in case of two colors, $M$ and $P$ are the same, while $B$ and $C$ are the same as well. It was conjectured in \cite{gkppvw} that $MB(n,2k+1,2)=O(n)$. Here we will prove this. Moreover, we extend all these results to more colors.

\begin{theorem} For any fixed $q$ and $c$, there exist an integer $n_0=n_0(q,c)$ and a constant $K=K(q,c)$ such that if $n\ge n_0$ holds, then we have $$MB(n,q,c), MC(n,q,c), PB(n,q,c), PC(n,q,c)\le Kn.$$

\end{theorem}

Note that for $n\le n_0$ there may not be a solution at all. For example if the answer is NO to every majority query, it only means every color class has size at most $q/2$, but does not help find an almost plurality ball.
We prove parts of this theorem separately. In Section 2 we prove $MB(n,2k+1,2)=O(n)$, in Section 3 we prove $MB(n,2k+1,c)=O(n)$ and $MB(n,2k,c)=O(n)$, while in Section 4 we prove $PB(n,q,c)=O(n)$. In all cases the proof implies the same bound for finding an almost plurality ball. We finish the paper with some remarks in Section 5.

\section{Large odd queries and two colors}

First we list the results of Gerbner, Keszegh, P\'alv\"olgyi, Patk\'os, Vizer and Wiener \cite{gkppvw} that we use. They introduced the following reformulation. The two colors are denoted by $0$ and $1$. In this case a median is always a non-minority ball.

\begin{theorem}[\cite{gkppvw}]\label{regi} \textbf{(i)} By asking all the queries of size $2k+1$, we can find a non-minority ball.

 \textbf{(ii)}  By asking all the queries of size 3, we can find a monotone ordering of the balls (note that there is no way to decide if the ordering is monotone increasing or decreasing).

 \textbf{(iii)}  If the answer to a query $\{x,y\}$ can be any of the two valid  statements out of $x\le y$, $y\le x$ or $x\neq y$, then by asking linearly many queries, we can find a median element.

\end{theorem}

We will also use the following simple algorithm from \cite{gkppvw}. We call it Algorithm WINNER-OUT. We start with a query of size $q$, set aside the answer ball and replace it by a new ball. Then repeat this procedure till all the balls are used (or we get a NO answer, if it is possible). After that, there is a set $S$ of $q-1$ balls never appearing as an answer ball. Observe that every ball not in $S$ belongs to one of the largest color classes in $S$. In particular, in case $q=2k+1$ and $c=2$, the set $S$ has the following property: either there are $k$ balls in $S$ of both colors, or every ball outside $S$ has the same color, and it is the majority color in $S$.

\begin{theorem}\label{fo}
$MB(n,2k+1,2)=O_k(n)$.
\end{theorem}

\begin{proof}
Observe that we can assume $n$ is large enough, as we are interested only in the order of magnitude.

First we apply Algorithm WINNER-OUT to obtain a set $S_1$. Then we apply the same to $X\setminus S_1$ to obtain $S_2$, and so on till we obtain $S_m$ with $m=8k^4$. This is the first phase of the algorithm. Let $S=\bigcup_{i=1}^m S_i$, then we have that a non-minority ball in $X\setminus S$ is a non-minority ball in $X$. Indeed, either every $S_i$ is balanced, or all the other balls are of the same color, and that is a non-minority color as $n$ is large enough.

From now on we will assume that every $S_i$ is balanced, then by the above it is enough to show a non-minority ball in $X\setminus S$ under this assumption. In the second phase of the algorithm we will find $k$ balls of one color and $k$ balls of the other color. We will only use queries that are completely inside $S$, thus constant many queries.

Let us first take an arbitrary $x\in S$, and remove it from $S$. Then we can find a non-minority ball $y$ in $S\setminus \{x\}$ using \textbf{(i)} of Theorem \ref{regi}. Then $y$ and $x$ have different colors, as $S$ is balanced. We apply this for every vertex, this way each ball has a pair of the other color. Let $G$ be the graph with edges corresponding to these pairs. 

\medskip 

\textbf{Case 1.} There is a ball $x$ with degree more than $2k^2$ in $G$. 

\smallskip

Without loss of generality $x\in S_1$. Then we delete $S_1$. There is a neighbor $y_1$ of $x$ in $S\setminus S_1$. Similarly to the above process, we can find a pair $x_1$ for $y_1$ in $S\setminus S_1$ by finding a non-minority ball in $(S\setminus S_1)\setminus \{y_1\}$ (we can do this, as $S\setminus S_1$ is balanced). Note that this pair may or may not be in $G$ as an edge.

Without loss of generality, $x_1$ is in $S_2$. Then we also delete $S_2$, $x$ has a neighbor $y_2$ among the remaining balls, and we can find a pair $x_2 \in (S\setminus S_1)\setminus S_2$ for $y_2$. We repeat this procedure multiple times. After deleting $S_{k-1}$, there is still at least one neighbor $y_{k-1}$ of $x$ remaining, and it has a pair $x_{k-1}$ among the remaining vertices. This way we found $k$ distinct balls of the color of $x$, and we have already found more than $2k^2$ balls of the other color, thus we are done with the second phase.

\medskip

\textbf{Case 2.} Every ball has degree at most $2k^2$ in $G$.

\smallskip

Observe that every component of $G$ is bipartite. If any component has at least $k$ vertices on both sides, we are done with the second phase again. Thus every component has at most $2k^3$ vertices, hence there are at least $k-1$ components. We take $k-1$ edges from $k-1$ different components. Let $A$ be the set of the $2k-2$ balls that are endpoints of these edges, and let $S'$ be the set we obtain by deleting from $S$ every $S_i$ that contains any ball from $A$. From the sets $S_j$ not deleted, we form $(8k^4-2k+2)/2$ disjoint pairs. For every such pair $S_{j_1}$ and $S_{j_2}$, and ask every query that contains $A$ and three balls from $S_{j_1}\cup S_{j_2}$. 

\smallskip

\textbf{Case 2.1.} For some pair $S_{j_1}$ and $S_{j_2}$, the answer is always one of the three balls from $S_i\cup S_j$.

Then we found a majority ball among those three, thus  we can find a monotone ordering of the balls in $S_i\cup S_j$ by \textbf{(ii)} of Theorem \ref{regi}. The first $k$ balls of this ordering are of the same color, and the next $k$ balls are of the other color, finishing the second phase. 

\smallskip

\textbf{Case 2.2.} For every pair $S_{j_1}$ and $S_{j_2}$, at least one of the answer balls is an element of $A$.

Then there is a ball $x$ in $A$ that occurs as an answer to at least $(8k^4-2k+2)/(4k-4)\ge 2k^3$ queries that all contain $A$ but are disjoint otherwise. Let $B$ be the set of $6k^3$ balls avoiding $A$ from $2k^3$ such queries, then there are at least $4k^3$ balls in $B$ of the color of $x$. In particular a non-minority ball $y_1$ of $B$ is of the color of $x$. Similarly, a non-minority ball $y_2$ in $B\setminus \{y_1\}$ also shares a color with $x$, and so on. This way we can find $2k^3$ balls of the color of $x$.
Their neighborhood in $G$ has size at least $k$, otherwise a vertex in their neighborhood would have degree more than $2k^2$. This way we found $k$ balls of the other color, finishing the second phase.

In the third phase, we denote the colors by $0$ and $1$, following \cite{gkppvw}. We are already given $k$ balls of color $0$ and $k$ balls of color $1$, from $S$. Let us consider two balls $x,y\in X\setminus S$. We ask two queries: $x$ and $y$ together with $k$ balls of color $0$ and $k-1$ balls of color $1$, and after that $x$ and $y$ together with $k-1$ balls of color $0$ and $k$ balls of color $1$. 

$\bullet$ Answer $x$ to the first query means $x\le y$, answer $y$ means $y\le x$, answer $0$ means at least one of $x$ and $y$ has color $0$, while answer $1$ means they both have color $1$, which implies $x\le y$. 

$\bullet$ Similarly, for the second query the answer tells either $x\le y$, or $y\le x$, or that at least one of $x$ and $y$ has color $1$. Combining the two answers, we obtain either $x\le y$, $y\le x$ or $x\neq y$. 

Thus \textbf{(iii)} of Theorem \ref{regi} finishes the proof.
\end{proof}

\section{Majority queries with more colors}

\begin{theorem} For any fixed $k$ and $c$, there exist an integer $n_0=n_0(k,c)$ and a constant $K=K(k,c)$ such that if $n\ge n_0$ holds, then we have $$MB(n,2k+1,c),MC(n,2k+1,c)\le Kn.$$
\end{theorem}

\begin{proof} Our algorithm consists of three phases. In the first phase we use Algorithm WINNER-OUT until we find a NO answer and then the first phase is over. If during Algorithm WINNER-OUT we get only YES answers, let $S$ be the set of balls never appearing as answer. Then either there are $k$ balls of one color and $k$ balls of another color in $S$, and all the other balls are of these two colors, or all the other balls have the same color. In both cases we can delete $S$, and find a non-minority ball among the other balls using $O(n)$ queries by Theorem \ref{fo}. That ball is a non-minority and thus an almost-plurality ball in $X$ as well.

Hence we can assume we found a NO answer to a query $Q$. In the second phase our goal is to find $k$ balls of one color, $k$ balls of another color and a ball of a third color. Let $q=2k+1$. We partition the set $X\setminus Q$ into subsets $A_1,\dots, A_{\lfloor(n-q+1)/cq\rfloor}$ of size $cq$ or $cq+1$. Then for every $i$ we ask all the $q$-subsets of $A_i\cup Q$. Observe that for a set $B\subset A_i\cup Q$ of size $k+1$, there is a majority ball in every query that contains $B$ if and only if $B$ is monochromatic. Indeed, for every color there are at least $k+1$ balls avoiding that color in $Q$, thus we can extend $B$ to a set of size $q$ without majority. Thus we can identify all the color classes of size at least $k+1$ inside $A_i\cup Q$.

If there are two color classes of size at least $k+1$ in $A_i\cup Q$ for some $i$, we can take $k$ balls from each and a ball from another color to obtain the desired configuration. Assume that for every $i$ there is only one such color class $B_i$ in $A_i \cup Q$. We are going to compare the color of $B_i$ and $B_j$. In order to do that, we take $k$ balls for $B_i$ and $k$ balls from $B_j$ and ask every query that consists of these $2k$ balls and one ball from $Q$.
The answer is always YES if and only if $B_i$ and $B_j$ are of the same color.

If we find $B_i$ and $B_j$ are of different colors, the query that was answered NO contains the desired configuration. Otherwise with linearly many queries we obtain that all the $B_i$'s are of the same color, and that is a non-minority, thus almost plurality color.

In the third phase we are given, say, $k$ blue balls, $k$ red balls and a green ball. Then asking 1 green, $k$ blue and $k-1$ red balls together with another ball $x$, the answer shows if $x$ is blue or not. Thus with $O(n)$ queries we find all the blue balls and similarly all the red balls. Among the remaining balls, we take a set $R$ of size $(c-2)k+1$, and ask all the queries consisting of $k$ blue balls and $k+1$ balls from $R$. As $R$ contains $k+1$ balls of the same color, we get a YES answer, finding a monochromatic set of a third color. Taking $k$ of those balls, $k-1$ blue balls and one red ball together with a ball that is neither blue, nor red, we can completely identify this third color class. Then similarly we can identify all the color classes, as long as there are at least $(c-2)k+1$ balls. If there are less than that many balls altogether in the remaining color classes, those cannot be non-minority/almost plurality if $n$ is sufficiently large, thus it is enough to check those classes we have already completely identified. This finishes the proof.
\end{proof}

\begin{theorem} For any fixed $k$ and $c$, there exist an integer $n_0=n_0(k,c)$ and a constant $K=K(k,c)$ such that if $n\ge n_0$ holds, then we have $$MB(n,2k,c),MC(n,2k,c)\le Kn.$$
\end{theorem}

\begin{proof} 


First we obtain a YES answer to a query $Q_0$, by asking all the $(2k)$-subsets of a set of size $ck+1$. Then we run Algorithm WINNER-OUT starting with $Q_0$; assume that the queries following $Q_0$ are $Q_1,Q_2,\dots$ and so on.

Let us consider first the case that Algorithm WINNER-OUT terminates after only YES answers. We claim that in this case all the answer balls are majority balls. Indeed, the answer balls to $Q_i$ and $Q_{i+1}$ have to be of the same color, as $Q_i$ contains at least $k+1$ balls of a color, then $Q_{i+1}$ contains at least $k$ balls of that color, thus it cannot be a minority color. Thus the $n-q+1$ answer balls are of the same color and they form majority.

Assume now that during Algorithm WINNER-OUT we obtain a NO answer, first to a query $Q_j$. Let the color of the answer ball $x$ to the previous query $Q_{j-1}$ be blue. Then $Q_{j-1}$ contains exactly $k+1$ blue balls. Let us ask all the queries of the form $Q_{j-1}\setminus \{x\} \cup \{y\}$ for $y\not\in Q_{j-1}$. The answer is NO if and only if $y$ is not blue. This way we can identify all the blue balls not in $Q_{j-1}$. Let $z$ be the ball in $Q_j\setminus Q_{j-1}$. Then we know $z$ is not blue. Let us ask the queries of the form $Q_{j-1}\setminus \{x\} \cup \{z\}$ for every $x\in Q_{j-1}$. The answer is YES if and only if $x$ is not blue. This way we can completely identify the blue color class.


If there are more than $n/2$ blue balls, we are done. Otherwise there is a $(ck+1)$-set $S$ of non-blue balls as $n$ is large enough. We ask all the $(k+1)$-subsets of $S$ together with $k-1$ blue balls. The answer is YES if and only if the $k+1$ balls have the same color. This way we identify one color class of size at least $k+1$, say red, in $S$. Then we ask all the queries consisting of $k-1$ blue balls, $k$ red balls and a ball of unknown color, to find all the red balls. Then repeat it with a third color, and so on. We can do that as long as there are more than $ck$ balls. Hence the number of unidentified balls is at most $ck$. Therefore, if $n$ is at least $c^2k$, we completely identified the two largest color classes, and that is enough to decide if there is a non-minority/almost plurality ball and show one such ball.
\end{proof}

\section{Plurality queries}

\begin{theorem}\label{plur} For any fixed $q$ and $c$, there exist an integer $n_0=n_0(q,c)$ and a constant $K=K(q,c)$ such that if $n\ge n_0$ holds, then we have $$PB(n,q,c), PC(n,q,c)\le Kn.$$
\end{theorem}

\begin{proof} We proceed by induction on $c$. Observe that if $c=2$, then majority and plurality queries are the same. Also, a ball is non-minority if and only if it is almost plurality. Therefore Theorem \ref{fo} yields the base case of the Theorem \ref{plur}.

The algorithm consists of multiple phases. In the first phase our goal is to find a query $Q$ that is answered NO and a ball $x\not\in Q$ that belongs to one of the largest classes in $Q$. In order to do that, first we find a YES answer to a query $Q_0$. It is doable by asking all the $q$-subsets of a set of size $qc$. Then we run algorithm WINNER-OUT starting with $Q_0$. If we get a NO answer to a query $Q$ at any point, the previously removed ball belongs to one of the largest color classes in $Q$. 

Assume all the answers are YES. Let $Q_1$ be the last query of algorithm WINNER-OUT,  and let $y$ be the answer ball to $Q_1$ and $S=Q_1\setminus\{y\}$. Let us call the color of $y$ blue. Observe that every ball not in $S$ belongs to one of the largest color classes in $S$.

Then we ask all the queries of the form $Q_1\cup \{u\}\setminus \{v\}$, where $v\in Q_1$, $u\not \in Q_1$. If $u$ is blue, then the answer is YES for every $v$. If $u$ belongs to another of the largest color classes in $S$, and there are at least three colors altogether, then the answer is not always YES. Hence if the answer is always YES, there are two possibilities. In both cases there are at most two colors outside $S$, and those colors have the same number of balls inside $S$, thus we can use Theorem \ref{fo} to find a non-minority, thus almost plurality ball in the entire set $X$ of balls.

Thus we can assume we obtain a NO answer to $Q=Q_1\cup \{u\}\setminus \{v\}$. Moreover, we can identify every blue ball outside $Q_1$ as a ball $x\in X\setminus Q_1$ is blue if and only if the answer to $Q_1\cup \{x\}\setminus \{v\}$ is yes for every $v\in Q_1$. If there is no such ball, we can delete $Q_1$. As all the colors appearing in balls outside $Q_1$ belong to the largest color classes in $S$, they have the same number of balls inside $Q_1$. As there is no blue ball in $X\setminus Q_1$, the number of colors used on balls in $X\setminus Q_1$ is at most $c-1$. So, by the induction hypothesis, we can find a non-minority and an almost plurality ball in $X\setminus Q_1$ with $O(n)$ queries. It is a non-minority/almost plurality ball in $X$, because blue cannot be a non-minority color, nor an almost plurality color since $n$ is large enough. Thus we can assume there is a blue ball $x$ outside $Q_1$.

In the second phase of the algorithm, our goal is to completely identify the color class of the ball $x$ that we found in the first phase. Let us call it blue. First we ask all the queries of the form $Q\cup\{x\}\setminus \{z\}$ with $z\in Q$. If $z$ is blue, the answer is obviously NO, while otherwise YES, thus we find all the blue balls inside $Q$. Now we ask the queries of the form $Q\cup\{y\}\setminus \{z\}$, where $y\not\in Q$ and $z\in Q$. If $y$ is blue, the answer is NO if and only if $z$ is blue. If $y$ is not blue, then the answer may still be NO in the case $z$ is blue (if there are at least 3 largest color classes in $Q$), but it is also NO in case $z$ is from another one of the largest color classes in $Q$. Thus we can find all the blue balls.

In case if there are less than $\lceil q/2\rceil$ blue balls altogether, we delete all the blue balls and repeat this argument to identify another color class. We can do this if we have $qc$ other balls. We repeat this until we identify a color with at least $\lceil q/2\rceil$ balls (it is doable if $n>qc+(c-1)\lceil q/2\rceil$). Without loss of generality we can assume that there are at least $\lceil q/2\rceil$ blue balls.


We continue with the third phase, where our goal is to find $\lceil q/2\rceil$ balls of another color class. We take a set $Z$ of $cq$ non-blue balls and we ask all of its subsets $T$ of size $\lceil q/2\rceil$ together with a set $B$ of $\lfloor q/2\rfloor$ blue balls. If $q$ is even, then the answer tells us if the $ q/2$ balls are monochromatic or not. This way we can completely identify the color classes of size at least $q/2$ inside $Z$. If $q=2l+1$ is odd, then the answer to $T\cup B$ is NO if and only if $T$ contains exactly $l$ balls of the same (non-blue) color. So if the answers are YES to all the queries $T\cup B$, then $Z$ is monochromatic, so again we are able to identify all large color classes of $Z$. 

Suppose that the answer to $T\cup B$ is NO and $T$ contains $l$ green balls. Then if a ball $z\in Z\setminus T$ is green, then for $t\in T$, $t$ is non-green if and only if the query $\{z\}\cup (T\setminus \{t\})\cup B$  answers YES (the answer ball must be in $\{z\}\cup (T\setminus \{t\})$); if $z$ is non-green, then for $t\in T$, $t$ is green if and only if the query $\{z\}\cup (T\setminus \{t\})\cup B$ answers YES (the answer ball must be in $B$). Then we can identify all green balls in $Z$. Considering all queries $T\cup B$ that answer NO, we can identify a color class of $Z$ that has size at least $\lfloor q/2\rfloor$ (in fact, we can identify all such color classes in $Z$).


In the fourth phase our goal is to identify all the balls of the color class, say red, that we found in the third phase. We ask $\lceil q/2\rceil-1$ red balls together with $\lfloor q/2\rfloor$ blue balls and a ball of unknown color. This tells if the unknown ball is red or not, completely identifying all the red balls.

After that we can similarly identify all the color classes, as long as there are at least $cq$ balls. If there are less than $cq$ balls altogether in the remaining color classes, those cannot be non-minority/almost plurality as $n$ is large enough, thus it is enough to check those classes we have already completely identified. This finishes the proof.
\end{proof}

\section{Concluding remarks}

Note that we did not try to optimize the constant factor in the algorithms presented in this paper. However, it is unlikely that our methods could lead to strong bounds.

There are results for majority and plurality problems where the algorithm has to be non-adaptive, i.e. all the queries have to be asked at the same time, before learning any answers. In case of two colors, a quadratic upper bound for queries of even size and a cubic lower bound for queries of odd size was given in \cite{gkppvw}.

Another variant studied in \cite{gkppvw} is the following. A ball is an $\alpha$-ball in $X$ if there are at least $\alpha(|X|-1)$ other balls of the same color in $X$. Thus a $1/2$-ball is a majority ball. This shows that if the answer to a query $Q$ is an $\alpha$-ball in $Q$ (or No if it does not exist), then we may not be able to find an $\alpha$-ball in the set of all balls. A natural question, studied in \cite{gkppvw}, is the following. For what $\beta$ can we find a $\beta$-ball using only $\alpha$-queries? It would be also interesting to extend these investigations to more colors.

\bigskip

{\large \textbf{Acknowledgement}}: The research of Huilan Chang was supported by the grant  MOST 106-2115-M-390 -004 -MY2. Research of D\'aniel Gerbner and Bal\'azs Patk\'os was supported by the National Research, Development and Innovation Office - NKFIH under the grants K 116769 and KH 130371, by the J\'anos Bolyai Research Fellowship of the Hungarian Academy of Sciences and the Taiwanese-Hungarian Mobility Program of the Hungarian Academy of Sciences.


\end{document}